\documentclass[12pt,a4paper]{amsart}
\usepackage{amsfonts,color}
\usepackage{amsthm}
\usepackage{amsmath}
\usepackage{amscd}
\usepackage[latin2]{inputenc}
\usepackage{t1enc}
\usepackage[mathscr]{eucal}
\usepackage{indentfirst}
\usepackage{graphicx}
\usepackage{graphics}
\usepackage{pict2e}
\usepackage{epic}
\usepackage{url}
\usepackage{epstopdf}

\numberwithin{equation}{section}
\usepackage[margin=2.6cm]{geometry}

\theoremstyle{plain}
\newtheorem{Th}{Theorem}[section]
\newtheorem{Lemma}[Th]{Lemma}

 \theoremstyle{definition}
\newtheorem{Def}[Th]{Definition}
\newtheorem{Conj}[Th]{Conjecture}
\newtheorem{Rem}[Th]{Remark}
\newtheorem{?}[Th]{Problem}

\newcommand{\e}{\varepsilon}
\newcommand{\prm}{\mathrm{pm}}

\begin{document}

\title{Covers, orientations and factors}

\author[P. Csikv\'ari]{P\'{e}ter Csikv\'{a}ri}

\address{MTA-ELTE Geometric and Algebraic Combinatorics Research Group \& ELTE: E\"{o}tv\"{o}s Lor\'{a}nd University \\ Mathematics Institute, Department of Computer 
Science \\ H-1117 Budapest
\\ P\'{a}zm\'{a}ny P\'{e}ter s\'{e}t\'{a}ny 1/C} 

\email{peter.csikvari@gmail.com}

\author[A. Imolay]{Andr\'as Imolay}

\address{ELTE: E\"{o}tv\"{o}s Lor\'{a}nd University \\ H-1117 Budapest
\\ P\'{a}zm\'{a}ny P\'{e}ter s\'{e}t\'{a}ny 1/C}

\email{imolay.andras@gmail.com}

\thanks{The first author was supported by the 
 Marie Sk\l{}odowska-Curie Individual Fellowship grant no. 747430, and before that grant he was partially supported by the
Hungarian National Research, Development and Innovation Office, NKFIH grant K109684 and Slovenian-Hungarian grant NN114614, and by the ERC Consolidator Grant  648017. The second author is partially supported by the New National Excellence Program (\'UNKP) and by the European Union, co-financed by the European Social Fund (EFOP-3.6.3-VEKOP-16-2017-00002).
}



\begin{abstract} Given a graph $G$ with only even degrees let $\varepsilon(G)$ denote the number of Eulerian orientations, and let $h(G)$ denote the number of half graphs, that is, subgraphs $F$ such that $d_F(v)=d_G(v)/2$ for each vertex $v$. Recently, Borb\'enyi and Csikv\'ari proved that  $\varepsilon(G)\geq h(G)$ holds true for all Eulerian graphs with equality if and and only if $G$ is bipartite. In this paper we give a simple new proof of this fact, and we give identities and inequalities for the number of Eulerian orientations and half graphs of a $2$-cover  of a graph $G$.

\end{abstract}

\maketitle

\section{Introduction}  In this paper we study the number of orientations and factors of Eulerian graphs.
Recall that a graph $G$ is Eulerian if every vertex of $G$ has even degree. In the literature it is often assumed that an Eulerian graph is also connected, but we will not require connectedness in this paper.
Let $\varepsilon(G)$ denote the number of Eulerian orientations, that is, the orientations where every vertex has in-degree equal to the out-degree.
Counting Eulerian orientations has triggered considerable interest  in combinatorics, computer science and statistical physics. Probably, the best known result is due to Lieb \cite{lieb2004residual} who determined the asymptotic number of Eulerian orientations of large grid graphs. Schrijver \cite{schrijver1983bounds} gave a lower bound on the number of Eulerian orientations in terms of the degree sequence. Welsh \cite{welsh1999tutte} observed that for a $4$--regular graph the Tutte-polynomial evaluation $|T_G(0,-2)|$  is exactly the number of Eulerian orientations since nowhere-zero $Z_3$-flows and Eulerian orientations are in one-to-one correspondence for $4$--regular graphs. Mihail and Winkler \cite{mihail1992number} gave an efficient randomized algorithm to sample and approximately count Eulerian orientations. 

Let $h(G)$ denote the number of half graphs, that is, subgraphs $F$ satisfying that $d_F(v)=d_G(v)/2$ for every vertex $v$.  Note that $h(G)>0$ if $G$ is not only Eulerian, but each of its connected component has an even number of edges. This condition is clearly necessary to have a half graph, and also sufficient: every second edge of an Eulerian tour will determine a half graph.

Recently, Borb\'enyi and Csikv\'ari \cite{borbenyi2020counting} proved that  $\varepsilon(G)\geq h(G)$ holds true for all Eulerian graphs with equality if and and only if $G$ is bipartite. This inequality fits into a series of inequalities comparing the number of orientations and subgraphs with the same property. For instance, it is known that the number of acyclic orientations is less than or equal to the number of acyclic subgraphs, that is, forests. The paper of Kozma and Moran \cite{kozma2013shattering} contains many more such inequalities, and sometimes equalities. In this paper we give a simple new proof of the fact $\varepsilon(G)\geq h(G)$, and we give identities and inequalities for the number of Eulerian orientations and half graphs of a $2$-cover of a graph $G$. In fact, we study the number of orientations and factors of $2$-covers with prescribed in-degree and degree sequences, respectively.

\subsection{Results.} For an edge set $A\subseteq E(G)$ let $\varepsilon(A)$ denote the number of Eulerian orientations of the graph $(V,A)$. Similarly, let $h(A)$ denote the number of half graphs of the graph $(V,A)$. Our first result is an identity for $\e(G)$ and $h(G)$.

\begin{Th} \label{recursion} Let $G$ be an Eulerian graph with edge set $E$. Then
$$\e(G)^2=\sum_{A\subseteq E}\e(A)\e(E\setminus A)$$
and
$$h(G)^2=\sum_{A\subseteq E}h(A)h(E\setminus A).$$
\end{Th}

Using the identities of Theorem~\ref{recursion} we can easily give a new proof of the following theorem of Borb\'enyi and Csikv\'ari \cite{borbenyi2020counting}.

\begin{Th}[Borb\'enyi and Csikv\'ari \cite{borbenyi2020counting}] \label{inequality} Let $G$ be an Eulerian graph. Then $\e(G)\geq h(G)$ with equality if and only if $G$ is bipartite. 
\end{Th}

Note that for a bipartite graph $G=(A,B,E)$ it is trivial that $\e(G)=h(G)$ since there is a natural correspondence between subgraphs and orientations: if an edge is oriented towards $B$, then put it into the subgraph, and if it is oriented towards $A$, then do not include it into the subgraph. This gives also a bijection between Eulerian orientations and half graphs.
\medskip 

As we will see it is natural to consider the number of Eulerian orientations and half graphs of $2$-covers of an Eulerian  graph $G$. 

\begin{Def} A $k$-cover (or $k$-lift) $H$ of a graph $G$ is defined as follows. The vertex set of  $H$ is $V(H)=V(G)\times \{0,1,\dots, k-1\}$, and if $(u,v)\in E(G)$,  then we choose a perfect matching between the vertices $(u,i)$ and $(v,j)$ for $0\leq i,j\leq k-1$. If $(u,v)\notin E(G)$, then there are no edges between $(u,i)$ and $(v,j)$ for $0\leq i,j\leq k-1$. 
\end{Def}

\begin{figure}[h!]
\begin{center}
\scalebox{.75}{\includegraphics{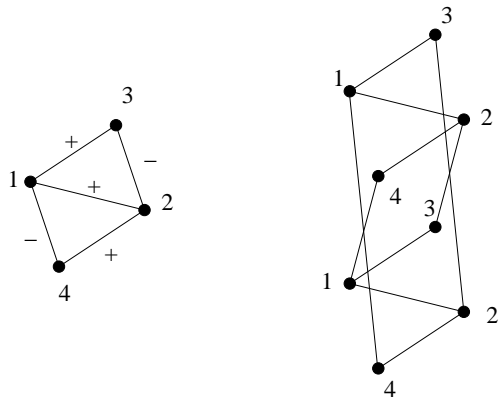}} 
\caption{A $2$-lift.} 
\end{center}
\end{figure}

When $k=2$ one can encode the $2$-lift $H$ by putting signs on the edges of the graph $G$: the $+$ sign means that we use the matching $((u,0),(v,0)),((u,1),(v,1))$ at the edge $(u,v)$,  
the $-$ sign means that we use the matching  $((u,0),(v,1)),((u,1),(v,0))$ at the edge $(u,v)$. For instance, if we put $+$ signs to every edge, then we simply get $G\cup G$ as $H$, and if we put $-$ signs everywhere, then the obtained $2$-cover $H$ is simply $G\times K_2$. (In general, the product graph $G\times H$ is defined as follows: $V(G\times H)=V(G)\times V(H)$ and the vertices $(u_1,v_1)$ and $(u_2,v_2)$ are adjacent if and only if $(u_1,u_2)\in E(G)$ and $(v_1,v_2)\in E(H)$.) Note that the graph $G\times K_2$ is bipartite, it is also called the bipartite double cover of $G$. Observe that if $G$ is bipartite, then $G\cup G=G\times K_2$, but other $2$-covers might differ from $G\cup G$.
\medskip

Graph cover techniques played important roles in the resolution of many open problems. Marcus, Spielmann and Srivastava \cite{marcus2013interlacing} used graph covers to construct Ramanujan graphs. The idea was suggested by Bilu and Linial \cite{bilu2006lifts}. Zhao \cite{zhao2010number} used the bipartite double cover to prove a conjecture of Alon \cite{alon1991independent} and Kahn \cite{kahn2001entropy} on the number of independent sets. Later he developed his ideas in the paper \cite{zhao2011bipartite}. Ruozzi \cite{ruozzi2012bethe} proved a conjecture of Sudderth, Wainwright, and Willsky \cite{willsky2008loop} on the Bethe approximation of an attractive graphical model by building on an observation due to Vontobel \cite{vontobel2013counting} connecting graph covers with Bethe approximation. Csikv\'ari \cite{csikvari2017lower} combined graph covers with graph limit theory to prove the so-called Lower Matching Conjecture of Friedland, Krop and Markstr\"om \cite{friedland2008number}. The properties of random lifts are also widely studied, see for instance the papers \cite{agarwal2019expansion} and \cite{greenhill2010number}.

\begin{Th} \label{bipartite cover} Let $G$ be an Eulerian graph with edge set $E$. Then
$$\e(G\times K_2)=h(G\times K_2)=\sum_{A\subseteq E}\e(A)h(E\setminus A).$$
\end{Th}

\noindent Combining Theorem~\ref{bipartite cover} with Theorem~\ref{inequality} we get that
$$\e(G\cup G)\geq \e(G\times K_2)=h(G\times K_2)\geq h(G\cup G)$$
since $\e(A)\geq h(A)$ and $\e(E\setminus A)\geq h(E\setminus A)$, and the equality holds true since $G\times K_2$ is a bipartite graph. This gives a slight refinement of Theorem~\ref{inequality} that already gave $\e(G\cup G)\geq h(G\cup G)$: we can see that we can put $\e(G\times K_2)$ between them.
\medskip

These inequalities can be generalized as follows. Let $H$ be an arbitrary $2$-cover of an Eulerian graph $G$. Then 
$$\e(G\cup G)\geq \e(H)\ \ \ \mbox{and}\ \ \ h(G\times K_2)\geq h(H).$$
In fact, even more general statement is true. To spell out this generalization we need the concepts $\e_{\underline{r}}(G)$ and $h_{\underline{r}}(G)$.

\begin{Def} Let $\underline{r}^G=(r_v)_{v\in V(G)}\in \mathbb{Z}^{V(G)}$. Let $\e_{\underline{r}}(G)$ denote the number of orientations of $G$ with in-degree $r_v$ at vertex $v$. We will call such an orientation an $\underline{r}$-orientation. 

Similarly, let $h_{\underline{r}}(G)$ be the number of subgraphs $F$ of $G$ with degree $d_F(v)=r_v$ for each vertex $v$. We will call such a subgraph an $\underline{r}$-factor.
\end{Def}

\noindent Clearly, if $r_v=d_G(v)/2$ for each vertex $v$, then $\e_{\underline{r}}(G)=\e(G)$ and $h_{\underline{r}}(G)=h(G)$. Other notable case is when $r_v=r$ for all $v$, then $h_{\underline{r}}(G)$ counts the number of $r$-factors.

\begin{Def} Let $\underline{r}^G=(r_v)_{v\in V(G)}\in \mathbb{Z}^{V(G)}$ and let $H$ be a $k$-cover of $G$. We say that $\underline{r}^H=(r_v)_{v\in V(H)}\in \mathbb{Z}^{V(H)}$ is induced by $\underline{r}^G$ if $r_{u'}=r_u$ for all lifts $u'\in V(H)$ of $u\in V(G)$. 
\end{Def}

\noindent In the following statements we can even drop the condition of $G$ being Eulerian.

\begin{Th} \label{2-cover inequality1} Let $\underline{r}^G=(r_v)_{v\in V(G)}\in \mathbb{Z}^{V(G)}$. Let $H$ be an arbitrary $2$-cover of the graph $G$. Let us denote by $\underline{r}$ the induced vector of $\underline{r}^G$ on $H$ and $G\cup G$. Then 
$$\e_{\underline{r}}(G\cup G)\geq \e_{\underline{r}}(H).$$
In other words, for any vector $\underline{r}\in \mathbb{Z}^{V(G)}$, the $2$-cover $H$ that maximizes $\e_{\underline{r}}(H)$
is $G\cup G$.
\end{Th}

\begin{Th} \label{2-cover inequality2} Let $\underline{r}^G=(r_v)_{v\in V(G)}\in \mathbb{Z}^{V(G)}$. Let $H$ be an arbitrary $2$-cover of a graph $G$. Let us denote by $\underline{r}$ the induced vector of $\underline{r}^G$ on $H$ and $G\times K_2$. Then 
$$h_{\underline{r}}(G\times K_2)\geq h_{\underline{r}}(H).$$
In other words, for any vector $\underline{r}\in \mathbb{Z}^{V(G)}$, the $2$-cover $H$ that maximizes $h_{\underline{r}}(H)$
is $G\times K_2$.
\end{Th}

\noindent Note that Theorems~\ref{2-cover inequality1} and \ref{2-cover inequality2} provide yet another proof of Theorem~\ref{inequality} by
$$\e(G)^2=\e(G\cup G)\geq \e(G\times K_2)=h(G\times K_2)\geq h(G\cup G)=h(G)^2,$$
where in the equality $\e(G\times K_2)=h(G\times K_2)$ we use only the fact that $G\times K_2$ is a bipartite graph, so we only use the trivial part of Theorem~\ref{inequality}.
\medskip

Next we generalize the concept of Eulerian orientations and half graphs. To every edge of $G$ let us
assign either $o$ or $s$, that is, orientation or subgraph. Then for all edge we have two choices: if we assigned $o$ to the edge, then we need to orient it, so when we consider the in-degree, we choose one of the end points and add $1$ to it, and add $0$ to the other endpoint.  If we assigned $s$ for the edge, then we need to decide whether we put this edge into subgraph or not, so we either add $1$ to the degrees of both endpoints, or add $0$ to the degrees of both endpoints. We will call such a configuration a factorientation. We will call the contribution of the edges to the vertex $v$ the mixed degree of $v$, that is, it is the sum of the in-degree from oriented edges and the degree coming from the subgraph. After we choose $o$ or $s$ for every edge, we say that a factorientation is balanced, if for every vertex $v$ the mixed degree is $d_G(v)/2$. Let $g(G)$ be the number of balanced factorientations.
We can see that this is a generalization of both Eulerian orientations and half graphs,
because if we assign $o$ to each edge, then $g(G)=\e(G)$ and if we assign $s$ to each edge, then
$g(G) = h(G)$. 

\begin{Th} \label{mixed 2-lift} Let $G$ be a graph with each edge equipped with one of the $4$ decorations $(+,o),(+,s),(-,o),(-,s)$.  Let $H$ be a $2$-cover of $G=(V,E)$ encoded by the $+$ and $-$ signs. Furthermore, we put the decorations $o$ and $s$ to the edges of $H$ consistently with the decoration of $G$, that is, 
for each edge $e\in E(G)$ the $2$-lifts of this edge get the same letter ($o$ or $s$). For an edge subset $B\subseteq E(G)$ let $\overline{B}$ denote the graph with vertex set $V(G)$ and edge set $B$, where for each edge $e\in B$ with a $-$ sign on it we swap the decorations $o$ to $s$ and vice versa. Then the number of balanced factorientations with respect to these decorations satisfies
$$g(H)=\sum_{A\subseteq E}g(A)g(\overline{E\setminus A}).$$
\end{Th}

Clearly, Theorem~\ref{mixed 2-lift} generalizes both Theorem~\ref{recursion} and \ref{bipartite cover}. We can also generalize Theorem~\ref{inequality} as follows.

\begin{Th} \label{mixed2} Let $G$ be an Eulerian graph with decorations $o$ and $s$ on the edges. Let $g(G)$ be the number of corresponding balanced factorientations. Then $g(G)\leq \e(G)$.
\end{Th}

With the method that we use to prove Theorem~\ref{recursion} we can also prove another nice result about orientations:

\begin{Th} \label{extra} Let $G$ be an Eulerian graph. Then $\varepsilon_{\underline{r}}(G)$ is maximal if $r_v=d_G(v)/2$ for all vertex $v$.
\end{Th}

\subsection{Organization of this paper and notations.} In the next section we prove Theorems~\ref{recursion}, \ref{inequality} and \ref{bipartite cover}. In Section~\ref{2-cover} we prove Theorem~\ref{2-cover inequality1} and Theorem~\ref{2-cover inequality2}. In Section~\ref{general 2-cover} we prove Theorems~\ref{mixed 2-lift} and \ref{mixed2}. In Section 5 we prove Theorem~\ref{extra}. We end the paper with an open problem.
\bigskip

\noindent \textbf{Notations.} For a graph $G$ and a vertex $v$ let $E_v$ denote the edges incident to $v$ in the graph $G$ and let $d_G(v)=|E_v|$. Furthermore, $N_G(v)$ denotes the set of neighbors of $v$. For a graph $G=(V,E)$ let $\prm(G)$ denote the number of perfect matchings. If $H$ is a $2$-cover of $G$, then $V_0$ and $V_1$ will denote the two copies of the vertex set. In particular, $G\times K_2$ is a bipartite graph with bipartite classes $V_0$ and $V_1$. For a vertex $u\in V(G)$ let $u_0$ and $u_1$ denote the two copies of $u$ in the $2$-cover $H$.

\section{Proof of Theorems~\ref{recursion}, \ref{inequality} and \ref{bipartite cover}}

In this section we prove Theorems~\ref{recursion}, \ref{inequality} and \ref{bipartite cover}.

\begin{proof}[Proof of Theorem~\ref{recursion}]
We only prove the identity $h(G)^2=\sum_{A\subseteq E}h(A)h(E\setminus A)$. The proof of the other identity is very similar. 
 Consider $S,T\subseteq E(G)$ determining two half graphs. Let $A=S\Delta T$ be the symmetric difference of the half graphs $S$ and $T$. Let $B=E\setminus A$, then $B=(S\cap T)\cup (E\setminus (S\cup T))$. Recall that $E_v$ denotes the edges incident to the vertex $v$. Since $|S\cap E_v|=|T\cap E_v|=d_G(v)/2$ we have $|E_v\cap (A\cap S)|=|E_v\cap (A\cap T)|$ and $|E_v\cap (S\cap T)|=|E_v\cap (E\setminus (S\cup T)|$. 
In other words, $A\cap S$ is a half graph of $A$, and $B\cap (S\cap T)$ is a half graph of $B$. Clearly, $S$ and $T$ determine $A,A\cap S,B, B\cap (S\cap T)$. But it is also true that $A,A\cap S,B, B\cap (S\cap T)$ determine $S$ and $T$ since $S=(A\cap S)\cup (S\cap T)$ and $T=(A\cap T)\cup (S\cap T)$. The number of quadruple $A,A\cap S,B, B\cap (S\cap T)$ is clearly $\sum_{A\subseteq E}h(A)h(E\setminus A)$. Hence
$$h(G)^2=\sum_{A\subseteq E}h(A)h(E\setminus A).$$
\medskip

In case of Eulerian orientations let $S$ and $T$ be two Eulerian orientations, and let $A$ be the set of edges where the orientations coincide and $B=E\setminus A$ be the remaining edges. Similarly to the previous discussion $S$ restricted to $A$ and $B$ gives an Eulerian orientation. The rest of the proof is the same.
\end{proof}

\begin{proof}[Proof of Theorem~\ref{bipartite cover}]
The proof is very similar to the proof of Theorem~\ref{recursion}. Since $G\times K_2=(V_0,V_1,E)$ is bipartite we have $\e(G\times K_2)=h(G\times K_2)$. Consider a half graph $S$ of $G\times K_2$. Let $\pi: G\times K_2\to G$ be the natural projection. For $k=0,1,2$ let  $A_k=\{e\in E\ |\ |\pi^{-1}(e)\cap S|=k\}$. Since $S$ was a half graph of $G\times K_2$ we get that for each vertex $v$ we have $|E_v\cap A_0|=|E_v\cap A_2|$. In other words, $A_0$ is a half graph of $A_0\cup A_2$. Let us orient an edge $(u,v)\in A_1$ from $u$ to $v$ if $(u_0,v_1)\in S$. Since $S$ was a half graph of $G\times K_2$ this gives an Eulerian orientation of the edges of $A_1$. Clearly, from the Eulerian orientation of $A_1$ and the half graph $A_0$ of $A_0\cup A_2$ we can immediately reconstruct $S$. Hence
$$h(G\times K_2)=\sum_{A_1\subseteq E}\e(A_1)h(A_0\cup A_2)=\sum_{A\subseteq E}\e(A)h(E\setminus A).$$

\end{proof}

\begin{proof}[Proof of Theorem~\ref{inequality}]
We prove the statement by induction on the number of edges of $G$. If the graph has no edge, then the statement is trivial. If it consists of a single cycle of length $k$ together with isolated vertices, then $\e(G)=2$ and $h(G)=0$ or $2$ depending on $k$ being odd or even. So in this case the theorem is true. If $G$ is different from a single cycle together with isolated vertices, then we use induction:
$$\e(G)^2-2\e(G)=\sum_{A\subseteq E \atop A\neq \emptyset ,E}\e(A)\e(E\setminus A)\geq \sum_{A\subseteq E \atop A\neq \emptyset ,E}h(A)h(E\setminus A)=h(G)^2-2h(G).$$
Since $\e(G)\geq 1$ and the function $x^2-2x$ is a monotone increasing function for $x\geq 1$ we get that $\e(G)\geq h(G)$.

In case of a bipartite graph we have equality by the bijection between subgraphs and orientations. If the graph $G$ is not bipartite, then it contains an odd cycle on an edge subset $A$ for which we have strict inequality $\e(A)>h(A)$. Then we have strict inequality in $\e(G)^2-2\e(G)>h(G)^2-2h(G)$ implying $\e(G)>h(G)$. Hence equality holds if and only if $G$ is bipartite.
\end{proof}

\section{Proof of Theorems~\ref{2-cover inequality1} and ~\ref{2-cover inequality2}} \label{2-cover}

In this section we prove Theorems~\ref{2-cover inequality1} and ~\ref{2-cover inequality2}.
The proofs rely on three observations, one of them is due to Schrijver relating the number of $\underline{r}$-orientations of a graph $G$ to the number of perfect matchings of a certain bipartite graph $G^*$ constructed from $G$. (In fact, we will slightly modify it, but still attribute it to Schrijver.) A similar observation connects the number of $\underline{r}$-factors of a graph $G$ to the number of perfect matchings of another graph $G^{**}$ constructed from $G$. The last observation is due to Csikv\'ari and gives an inequality between the number of perfect matchings of certain $2$-covers of a graph $G$.

\begin{Lemma}[Schrijver \cite{schrijver1983bounds}] \label{Euler-matchings} Let $G$ be a graph, and let $G^*$ be the following bipartite graph. On one side of the bipartite graph every vertex corresponds to an edge $e\in E(G)$. On the other side of the bipartite graph we take $r_v$ copies of each vertex $v$. Finally, an edge $e=(u,v)$ is adjacent to all copies of $u$ and $v$. Then
$$\prm(G^*)=\e_{\underline{r}}(G)\prod_{v\in V(G)}r_v!.$$ 
\end{Lemma}

\begin{proof} Let $R=\prod_{v\in V(G)}r_v!$. There is an $1$ to $R$ map from the set of $\underline{r}$-orientations to the set of perfect matchings of $G^*$. Namely, if the edges $(u,v_i)\in E(G)$ for $i=1,\dots r_u$ are oriented towards $u$, then we take the union of  perfect matchings between $e_{uv_i}$ and the $r_u$ copies of $u$ to get a perfect matching of $G^*$.
\end{proof}

\noindent A similar lemma enable us to encode $h_{\underline{r}}$ via perfect matchings. A qualitative version of this lemma appeared in \cite{tutte1954short}.

\begin{Lemma}[Tutte \cite{tutte1954short}] \label{r-factors} Let $G$ be a graph, and let $G^{**}$ be the following graph. For each edge $e=(u,v)$ we introduce two vertices $e_{uv}$ and $e_{vu}$, and for each vertex $v$ we introduce $r_v$ copies of $v$. Then we connect $e_{uv}$ with $e_{vu}$, and we also connect the $r_u$ copies of $u$ with $e_{vu}$ for each $v\in N_G(u)$. Then
$$\prm(G^{**})=h_{\underline{r}}(G)\prod_{v\in V(G)}r_v!.$$
\end{Lemma}

\begin{proof} Let $R=\prod_{v\in V(G)}r_v!$. Again there is an $1$ to $R$ map from the set of $\underline{r}$-factors to the set of perfect matchings of $G^{**}$. Namely, if the edges $(u,v_i)\in E(G)$ for $i=1,\dots r_u$ are in the $\underline{r}$-factor, then to get a perfect matching of $G^{**}$, we take the union of  perfect matchings between $e_{v_iu}$ and the $r_u$ copies of $u$ together with those edges $(e_{xy},e_{yx})$ for which $(x,y)\in E(G)$, but is not in the $\underline{r}$-factor.
\end{proof}

\noindent Next we need a lemma that relates perfect matchings with covers.

\begin{Lemma}[Csikv\'ari \cite{csikvari2017lower,csikvari2016extremal}] \label{2-cover matchings} Let $G$ be a graph, and let $H$ be an arbitrary $2$-cover of $G$. Then
$$\prm(H)\leq \prm(G\times K_2).$$
In particular, if $G$ is bipartite, then $\prm(H)\leq \prm(G)^2$.
\end{Lemma}

\begin{proof}[Sketch of the proof]
Let us project the edges of a perfect matching of a $2$-cover $H$ to the graph $G$. The obtained configuration consists of cycles and double-edges, that is, two edges projected to the same edge. (Every degree must be $2$ in the obtained configuration.) For such a configuration we can count the number of preimages. Each cycle can be lifted in at most $2$ ways since the preimage of one edge determines the preimage of the subsequent edges in the cycle. It may occur though that we cannot close the cycle. (This happens for instance if try to lift a $3$-cycle in the union of two $3$-cycles.) On the other hand, it is easy to see that on $G\times K_2$ every cycle can be lifted in exactly $2$ ways. This means that every configuration has at least as many preimages on $G\times K_2$ than on another $2$-cover $H$.

\end{proof}

\begin{proof}[Proof of Theorem~\ref{2-cover inequality1}]
Let $H$ be an arbitrary cover of the graph $G$. Let us construct the bipartite graphs $G^*$ and $H^*$ of Lemma~\ref{Euler-matchings}. Observe that $H^*$ is also a $2$-cover of $G^*$, and so by Lemma~\ref{2-cover matchings} we have $\prm(G^*)^2\geq \prm(H^*)$. Using Lemma~\ref{Euler-matchings} we have
$$\prm(G^*)^2=\e_{\underline{r}}(G)^2\left(\prod_{v\in V(G)}r_v!\right)^2$$
and
$$\prm(H^*)=\e_{\underline{r}}(H)\prod_{v\in V(H)}r_v!=\e_{\underline{r}}(H)\left(\prod_{v\in V(G)}r_v!\right)^2.$$
Hence $\e_{\underline{r}}(G)^2\geq \e_{\underline{r}}(H)$.
\end{proof}

\begin{Rem} An interesting application of Theorem~\ref{2-cover inequality1} is the following.
Let $T_{n,m}$ be the toroidal grid of size $n\times m$, that is, a grid of size $n\times m$ closed in a toroidal way to make it $4$-regular. Then $\e(T_{n,m})\geq \left(\frac{4}{3}\right)^{3nm/2}$. This can can be seen as follows. Lieb \cite{lieb2004residual} showed that $\lim_{n,m\to \infty}\e(T_{n,m})^{1/nm}=\left(\frac{4}{3}\right)^{3/2}$. On the other hand, $T_{2n,m}$ and $T_{n,2m}$ are $2$-covers of $T_{n,m}$ so $\e(T_{n,m})^2\geq \e(T_{2n,m})$ and $\e(T_{n,2m})$ so it is necessary that $\e(T_{n,m})^{1/nm}\geq \left(\frac{4}{3}\right)^{3/2}$ for every $n,m$.
\end{Rem}

\begin{proof}[Proof of Theorem~\ref{2-cover inequality2}]
Let $H$ be an arbitrary cover of the graph $G$. Let us construct the  graphs $G^{**}$ and $H^{**}$ of Lemma~\ref{r-factors}. Observe that $H^{**}$ is also a $2$-cover of $G^{**}$, and $(G\times K_2)^{**}=G^{**}\times K_2$ and so by Lemma~\ref{2-cover matchings} we have $\prm((G\times K_2)^{**})\geq \prm(H^{**})$. Using Lemma~\ref{r-factors} we have
$$\prm((G\times K_2)^{**})=h_{\underline{r}}(G\times K_2)\prod_{v\in V(G\times K_2)}r_v!=h_{\underline{r}}(G\times K_2)\left(\prod_{v\in V(G)}r_v!\right)^2$$
and
$$\prm(H^{**})=h_{\underline{r}}(H)\prod_{v\in V(H)}r_v!=h_{\underline{r}}(H)\left(\prod_{v\in V(G)}r_v!\right)^2.$$
Hence $h_{\underline{r}}(G\times K_2)\geq h_{\underline{r}}(H)$.

\end{proof}

\section{general 2-cover} \label{general 2-cover}

In this section we prove Theorems~\ref{mixed 2-lift} and ~\ref{mixed2}.

\begin{proof}[Proof of Theorem~\ref{mixed 2-lift}]

Consider a balanced factorientation $S$ of the graph $H$. Take the natural projection from $H$ to $G$, and let $A$ be the
set of edges for which the projected edges coincide, that is, both edges are oriented in the same way if there is an $o$ on that edge, or they are both or neither in the subgraph if there is an $s$ on that edge.

Let $B=E\setminus A$. The factorientation $S$ was balanced so for all
vertex $v$ the mixed degrees of $v_0$ and $v_1$ are both $d_G(v)/2$. This means that  after the natural projection -- and the doubling of the original edges of the graph-- the mixed degree of vertex $v$ is $d_G(v)$. If an edge is in $A$, then it contributes either $0$ or $2$ to the mixed degree of $v$, otherwise it contributes $1$. Thus there must be equal number of $0$'s and $2$'s contributions, which means that if we restrict the graph to $A$, then we also get a balanced factorientation. 

For an edge $(u, v)\in B$ if $(u,v)$ has a plus sign on it, then it contributes  the same amount to the mixed degree of $u$ and $v$ as it contributed to  the mixed degrees of $u_0$ and $v_0$ with the edge $(u_0,v_0)$ in $S$. 

If $(u,v)$ has a minus sign on it, then in $(u,v)$ change $s$ to $o$ and vice verse. If it was $o$ before and the orientations were $u_0\to v_1$ and $v_0\to u_1$, then we do not put the edge  $(u,v)$ into the subgraph. If it was $o$ before and the orientations were $u_1\to v_0$ and $v_1\to u_0$, then we do  put the edge  $(u,v)$ into the subgraph. Note that the contribution of these edges to the mixed degree of a vertex $u$ is the same as the contribution of the original edges to the vertex $u_0$.

Similarly, if it was an $s$ before, and $(u_0,v_1)$ was in the subgraph, but $(u_1,v_0)$ was not, then orient the edge $(u,v)$ from $v$ to $u$. This way it is again true that the contribution of these edges  to the mixed degree of a vertex $u$ is the same as the contribution of the original edges to the vertex $u_0$.

So for all $(u, v)\in B$ we made sure that the contribution of the orientation or factor to the mixed degree of a vertex $u$ is the same contribution as the original edges to the mixed degree of the vertex $u_0$. Finally, observe that since the mixed degree of $u_0$ and $u_1$ was the same, and the edges of $A$ contributed the same to the mixed degrees, it is necessary that the mixed degree contributed by the edges of $B$ to the vertex $u$ is exactly $d_B(u)/2$. This means that the constructed factorientation is balanced if we restrict to the edges of $B$.

Finally, observe that if we get a balanced factorientation of $A$ and $\overline{E\setminus A}$ then we can easily get back the balanced factorientation of $H$. This finishes the proof. 
\end{proof}

\begin{proof}[Proof of Theorem~\ref{mixed2}]
The proof is practically the same as the proof of Theorem~\ref{inequality}. We use induction to the number of edges. We have
$$g(G)^2=g(G\cup G)=\sum_{A\subseteq E}g(A)g(\overline{E\setminus A})=\sum_{A\subseteq E}g(A)g(E\setminus A)$$
since $G\cup G$ corresponds to the $2$-cover with only $+$ signs. By induction we have \\ $g(A)\leq \e(A)$ if $A\neq E$. Hence $g(G)^2-2g(G)\leq \e(G)^2-2\e(G)$ which implies that $g(G)\leq \e(G)$.
\end{proof}

\section{Proof of Theorem~\ref{extra}}

In this section we prove Theorem~\ref{extra}.

\begin{proof}[Proof of Theorem~\ref{extra}]
We prove the statement by induction on the number of edges. The statement is true for cycles.

Let $\underline{r}\in \mathbb{Z}^{V(G)}$ be an arbitrary vector and let $\underline{d} \in \mathbb{Z}^{V(G)}$ be the vector where $d_v=d_G(v)$. Notice that $\e_{\underline{r} }(G)=\e_{\underline{d}-\underline{r} }(G)$, because if we change the direction of all edges in an $\underline{r}$-orientation, then we get a $(\underline{d}-\underline{r})$-orientation. 
Consider $S$ and $T$ determining an $\underline{r}$-orientation and a $(\underline{d}-\underline{r})$-orientation, respectively. Let $A \subseteq E$ be the set of edges where the two orientations coincide, and let $B=E \setminus A$ be the remaining edges. We claim that $A \cap S$ is an Eulerian orientation of $A$. To see this consider a vertex $v$ of $G$. Let $d_S^A(v)$ be the in-degree of $v$ in the orientation $S$ restricted to the set $A$. We similarly define $d_S^B(v)$, $d_T^A(v)$ and $d_T^B(v)$. Finally, let $d^A(v)$ be the degree of $v$ in the graph $(V,A)$. We similarly define $d^B(v)$. Then the in-degree of $v$ in the orientation $S$ is $r_v=d_S^A(v)+d_S^B(v)$. The in-degree of $v$ in the orientation $T$ is $d_v-r_v=d_T^A(v)+d_T^B(v)$. By the definition of $A$ we have $d_S^A(v)=d_T^A(v)$.  By the definition of $B$ we have $d_S^B(v)=d^B(v)-d_T^B(v)$. Hence
\begin{align*}
d^A(v)&=d_v-d^B(v)\\
      &=r_v+(d_v-r_v)-d^B(v)\\
			&=d_S^A(v)+d_S^B(v)+d_T^A(v)+d_T^B(v)-d^B(v)\\
			&=d_S^A(v)+d_T^A(v)\\
			&=2d_S^A(v).
\end{align*}
Since this is true for all vertex $v$ of $G$, we get that $A \cap S$ is an Eulerian orientation of $A$. From this it is easy to see that $S \cap B$ is an $(\underline{r}-\underline{d_A}/2)$-orientation of $B$. If we choose $S \cap B$ that determines $T \cap B$,  thus we get that 
$$\e_{\underline{r} }(G)\e_{\underline{d}-\underline{r} }(G)=\sum_{A\subseteq E}\e(A)\e_{\underline{r}-\underline{d_A}/2 }(E\setminus A)$$
The degree of a vertex $v$ in $E \setminus A$ is $d_G(v)-d_A(v)$ so from induction we know that $\e_{\underline{r}-\underline{d_A}/2 }(E \setminus A) \leq \e_{\underline{d}/2-\underline{d_A}/2 } (E \setminus A)$ if $A \neq \emptyset$. Thus
\begin{align*}
\e_{\underline{r} }(G)\e_{\underline{d}-\underline{r} }(G)-\e_{\underline{r} }(G)&=\sum_{A\subseteq E \atop A \neq \emptyset}\e(A)\e_{\underline{r}-\underline{d_A}/2 }(E\setminus A)\\
&\leq \sum_{A\subseteq E \atop A \neq \emptyset}\e(A)\e_{\underline{d}/2-\underline{d_A}/2 }(E\setminus A)\\
&=\e_{\underline{d}/2 }(G)^2-\e_{\underline{d}/2 }(G).
\end{align*}
In the last step we used Theorem~\ref{recursion}. Recall that  $\e_{\underline{r}}(G)=\e_{\underline{d}-\underline{r} }(G)$. Hence $\e_{\underline{r} }(G)^2-\e_{\underline{r} }(G)\leq \e_{\underline{d}/2 }(G)^2-\e_{\underline{d}/2 }(G)$. This implies that $\e_{\underline{r}}(G)\leq \e_{\underline{d}/2 }(G)=\e(G)$.
\end{proof}

\section{Open problem}

We end this paper with an open problem.

\begin{Conj}
Let $G$ be an Eulerian graph, and let $H$ be a $k$-cover of $G$. Then $\e(G)^k\geq \e(H)$.
\end{Conj}

\noindent \textbf{Acknowledgments.} The first author thank L\'aszl\'o Kozma for calling attention to the paper \cite{kozma2013shattering}. The authors also thank the reviewers for helpful comments.

\bibliographystyle{siamnodash}

\bibliography{covers_Eulerian_orientations_half_graphs}

\end{document}